\documentclass[10pt,a4paper]{article}
\usepackage{amsmath}
\usepackage{amsfonts}
\usepackage{amssymb}
\usepackage{graphicx}
\usepackage{amsthm}
\usepackage[all]{xy}

\usepackage[subnum]{cases}

\newcounter{oftheorem}[subsection]
\newenvironment{mytheorem}[1]%
{\begin{trivlist}
     \renewcommand{\theoftheorem}{#1~\thesection.\arabic{oftheorem}}
     \refstepcounter{oftheorem}
     \item[\hspace{\labelsep}\bf\thesection.\arabic{oftheorem} #1]}%
{\end{trivlist}}
\newenvironment{definition}{\begin{mytheorem}{Definition}\it}{\end{mytheorem}}

\newenvironment{proposition}{\begin{mytheorem}{Proposition}\it}{\end{mytheorem}}
\newenvironment{theorem}{\begin{mytheorem}{Theorem}\it}{\end{mytheorem}}

\newenvironment{remark}{\begin{mytheorem}{Remark}}{\end{mytheorem}}
\newenvironment{lemma}{\begin{mytheorem}{Lemma}}{\end{mytheorem}}

\author{Kostas Katsios, Stavros Anastassiou\\
Center of Research and Applications of Nonlinear Systems,\\ (CRANS),\\
 University of Patras, Department of Mathematics,\\ GR-26500 Rion, Greece}

\title{Darboux polynomials and global phase portraits for the $D_2$ vector field}
\begin{document}
\maketitle
\begin{abstract}
We study a vector field of $\mathbb{R}^3$ equivariant under the $D_2$ symmetry group, called ``the $D_2$ field'' in the literature. We construct the complete list of Darboux polynomials for it, solving the partial differential equation defining them. We also use these polynomials to comment on its global qualitative behaviour. This is meant to be a first step towards the comparison of vector fields based on the module generated by their Darboux polynomials.   
\end{abstract}
\section{Introduction}
Vector fields possessing some kind of symmetry occur very often in applications and attract therefore the attention of many researchers. For an introduction to the problems and methods of dynamical systems 
invariant under a symmetry group one may consult \cite{field1,field2,chossat}, while in \cite{lett} an exposition is given of the 
topological properties of three--dimensional flows invariant under 
various symmetry groups.

Let us recall that, if $(G,*)$ is a compact Lie group, a polynomial $p\in \mathbb{R}[x_1,x_2,..,x_n]$ is called invariant with respect to this group if $p(gx)=p(x), \forall g\in G$. The set of all invariant polynomials is called the invariant ring of $G$ and it is generated by finitely 
many homogeneous invariant polynomials.

A polynomial mapping $f:\mathbb{R}^n\rightarrow \mathbb{R}^n$ is called equivariant if $f(gx)=gf(x),\ \forall g\in G$. Their set forms a module 
over the invariant ring. 

In what follows we do not distinguish between a vector field $f:\mathbb{R}^3\rightarrow \mathbb{R}^3$ 
and the associated system of equations $\dot{x}=f(x),x\in \mathbb{R}^3$. We shall 
call symmetric the system of equations associated with an equivariant vector field.

Here we consider the group generated by the following transformations of $\mathbb{R}^3$:
\begin{center}
$(x,y,z)\mapsto (x,y,z)$, $(x,y,z)\mapsto (-x,-y,z)$,\\
$(x,y,z)\mapsto (-x,y,-z)$, $(x,y,z)\mapsto (x,-y,-z)$.
\end{center}

The following was proved in \cite{D2}.
\begin{lemma}
A three--dimensional autonomous system, symmetric with respect to the above group and having linear and quadratic terms, can be written in the form:
\begin{eqnarray} \label{d2}
\dot{x}&=&ax+yz \nonumber \\
\dot{y}&=&by+xz \\
\dot{z}&=&z\pm xy \nonumber
\end{eqnarray}
where $a,b$ are free parameters.
\end{lemma}
If the sign of the nonlinear term in the third equation is chosen to be positive, the system is gradient and its behaviour quite simple to analyse. We shall therefore be interested in system (\ref{d2}) only in the case that the nonlinear term of the third equation is negative (we will, however, include some results for the positive sign case). Following 
\cite{D2}, we shall call this system ``$D_2$'' and we denote the corresponding vector field by $X$.

There are two reasons for studying vector field $X$. First, it is the representative of an entire class of vector fields, namely quadratic fields equivariant under this specific symmetry group. Thus, be studying this field we illuminate the behaviour of its entire class. Second, its global behaviour is quite complicated, as demonstrated in \cite{Anastassiou}. Studying its phase space provides an understanding of the interplay between symmetry and complicated structures arising in it. 

We wish to study the existence of Darboux polynomials of this vector field. A polynomial $f:\mathbb{R}^n\rightarrow \mathbb{R}$ is called a Darboux polynomial for the vector field $X$ of $\mathbb{R}^n$ if it satisfies equation $\mathcal{L}_X f=kf$, where $\mathcal{L}_X f$ denotes the Lie derivative of $f$ along the flow
of the field $X$ and $k$ is a real polynomial $k:\mathbb{R}^n\rightarrow \mathbb{R}$, called the cofactor of $f$. In this case, the algebraic manifold $f=0$ remains invariant under the flow of $X$. If the cofactor is 
identically zero, $f$ is a polynomial first integral for $X$.

Knowledge of the Darboux polynomials a vector field has is of great importance, since they can be used to study the field's global phase space. One may see \cite{Tinghua2007,Llibre-Valls2011,Llibre2013,Llibre-Valls2015} for studies conserning this direction.

Here, in section \ref{darboux}, we construct the complete list of Darboux polynomials for the vector field $X$, for every value of its parameters. To achieve that, we solve equation $\mathcal{L}_Xf=kf$, exploiting the degrees of the polynomials involved. We also present, in the same section, the complete list of Darboux polynomials for vector field (\ref{d2}), in the case of the positive sign in the third equation. In section \ref{portraits}, we present the global phase portraits of $X$, using the Darboux polynomials already constructed. Some of the results of this section have already appeared in \cite{Anastassiou}. However, we include them here as well, in a more systematic way, both for the sake of completeness and to be able to comment on them.

Let us here note that in \cite{Llibre-Zhang} the Darboux polynomials for the R\"{o}ssler system were presented, a system wich possesses no symmetry, while in \cite{Swinnerton} the complete list of Darboux polynomials for the Lorenz system was constructed, a system which possesses a two--fold symmetry. Our ultimate goal is to detect symmetries (or other qualitative properties) of vector fields based on the algebraic properties of the module generated by their Darboux polynomials. The study of the Darboux polynomials of a system symmetric with respect to the $D_2$ group serves as a first step towards this direction. 

\section{Darboux polynomials for the vector field $X$}
\label{darboux}
In this section we are interested in obtaining the Darboux polynomials for the vector field $X=(ax+yz)\frac{\partial}{\partial x}+(by+xz)\frac{\partial}{\partial y}+(z-xy)\frac{\partial}{\partial z}$, introduced above. To accomplish that, we begin by studying the cofactor such a Darboux polynomial may have.

The degree of a monomial $C x^{\lambda} y^{\mu} z^{\nu}$ is $d=\lambda + \mu + \nu$, for $\lambda,\mu,\nu \in \mathbb{N}$ and $C \in \mathbb{R}$, and the degree of a polynomial is the largest of the degrees of its constituent monomials. If $f:\mathbb{R}^3\rightarrow \mathbb{R}$ is a Darboux polynomial for the vector field X, equation $\mathcal{L}_Xf=kf$ ensures that the cofactor $k:\mathbb{R}^3\rightarrow \mathbb{R}$ is a polynomial of degree at most 1. Actually, as we show in the next lemma, one may concentrate on constant cofactors.
\begin{lemma}
If there is no Darboux polynomial for $X$ having constant cofactor, $X$ does not have Darboux polynomials.
\end{lemma}
\begin{proof}
We use Lemma 2.2 of \cite{Ferragut-Gasull2015}. Since vector field $X$ is equivariant under the linear mapping $\sigma (x,y,z)=(-x,-y,z)$, if $f(x,y,z)$ is a Darboux polynomial for $X$ having cofactor $k(x,y,z)$, the polynomial $f(\sigma (x,y,z))$ is also a Darboux polynomial for $X$ with cofactor $k(\sigma (x,y,z))$, as a direct computation shows. Indeed:
\begin{center}
$\mathcal{L}_Xf(\sigma (x,y,z))=D(f(\sigma (x,y,z)))\cdot X(x,y,z)=$\\
$=\nabla f(\sigma (x,y,z))D\sigma (x,y,z)\cdot X(x,y,z)=\nabla f(\sigma (x,y,z))\sigma(X(x,y,z))=$\\
$=\nabla f(\sigma (x,y,z))X(\sigma(x,y,z))=\mathcal{L}_{X(\sigma(x,y,z))}f(\sigma (x,y,z))=$\\
$=k(\sigma(x,y,z))f(\sigma (x,y,z))$.
\end{center}
Similar calculations show that $f(x,y,z)f(\sigma (x,y,z))$ is also a Darboux polynomial for $X$, having cofactor $k(x,y,z)+k(\sigma (x,y,z))$. It follows that, if $k(x,y,z)=c_0+c_1x+c_2t+c_3z$ is a cofactor of some Darboux polynomial, $k(x,y,z)+k(-x,-y,z)=c_0+c_3z$ is also a Darboux polynomial. 

Using the symmetry $(x,y,z)\mapsto (x,-y,-z)$ one can get rid of the term $c_3z$ as well. Thus, if a Darboux polynomial exists with cofactor $k(x,y,z)$, there exists a Darboux polynomial with a constant cofactor.  
\end{proof}
Due to this lemma, we shall proceed seeking for Darboux polynomials having constant cofactor.

Assume now that $$f(x,y,z)=\sum _{j=0}^df_j(x,y,z)$$ is a Darboux polynomial of $X$, where $f_j$ are homogeneous polynomials of degree $j$. Studying equation $\mathcal{L}_Xf=kf$, where as we have shown $k(x,y,z)=c_0$, we arrive at the following:
\begin{lemma}
Polynomial $f(x,y,z)$ is a Darboux polynomial for $X$ with constant cofactor $c_0$ if, and only if, the following equations are satisfied:
\begin{subequations}
\begin{align}
\label{eqn2solve1}
yz\frac{\partial f_d}{\partial x}+xz\frac{\partial f_d}{\partial y}-xy\frac{\partial f_d}{\partial z} &= 0\\
\label{eqn2solve2}
ax\frac{\partial f_j}{\partial x}+by\frac{\partial f_j}{\partial y}+z\frac{\partial f_j}{\partial z}+yz\frac{\partial f_{j-1}}{\partial x}+xz\frac{\partial f_{j-1}}{\partial y}-xy\frac{\partial f_{j-1}}{\partial z} &= c_0f_j\\
\label{eqn2solve3}
c_0f_0 &=0
\end{align}
\end{subequations}
where $j=1,..,d$.
\end{lemma}
Thus, to compute the Darboux polynomials of $X$ with constant cofactor, we have to solve the above system of equations. We begin by solving the first of these three equations using the method of characteristics.
\begin{lemma}
The solutions of equation (\ref{eqn2solve1}) are of the form 
\begin{center}
$f_d(x,y,z)=G(x^2-y^2,x^2+z^2)$,
\end{center}
where $G:\mathbb{R}^2\rightarrow \mathbb{R}$ a homogeneous polynomial of degree $d/2$.
\end{lemma}
\begin{proof}
The characteristics of this p.d.e. are the solutions of system:
\[
\begin{cases}
\dot{x}=yz\\
\dot{y}=xz\\
\dot{z}=-xy
\end{cases},
\]
and it is easy to see that $u=x^2-y^2,\ v=x^2+z^2$ are two first integrals for this system. Polynomial $f_d$ is constant on the characteristics, thus the conclusion.
\end{proof}
Notice that we have also showed that $d$, that is the degree of the Darboux polynomial we are after, should be an even number.

We next proceed to show that the Darboux polynomial $f$ as no terms of degree $d-1$.
\begin{lemma}
Polynomial $f$ does not contain any terms of order $d-1$.
\end{lemma}
\begin{proof}
Equation (\ref{eqn2solve2}), for $j=d$ and $f_d$ the polynomial we computed above, becomes:
\begin{center}
$yz\frac{\partial f_{d-1}}{\partial x}+xz\frac{\partial f_{d-1}}{\partial y}-xy\frac{\partial f_{d-1}}{\partial z}=c_0f_d-ax\frac{\partial f_d}{\partial x}-by\frac{\partial f_d}{\partial y}-z\frac{\partial f_d}{\partial z}$.
\end{center}

The characteristic equations are the same as before, and we also get the same first integrals. 

We consider thus the following transformation of variables:
\begin{center}
$u=x^2-y^2,\ v=x^2+z^2,\ w=z$,
\end{center}
with inverse
\begin{center}
$x=\pm \sqrt{v-w^2},\ y=\pm \sqrt{v-w^2-u},\ z=w$.
\end{center}
Rewriting the pde above in the variables $u,v,w$ and keeping $u,v$ fixed, we arrive at the following ode:
\begin{center}
$\frac{df_{d-1}}{dw}(w)=\pm \frac{c_0 G(u,v)}{\sqrt{v-w^2-u}\sqrt{v-w^2}}$,
\end{center}
the solution of which is:
\begin{center}
$f_{d-1}(w)=c\pm \frac{c_0EI(\arcsin(\sqrt{v^{-1}}w),\frac{v}{-u+v})}{\sqrt{v-u}}$,
\end{center} 
and thus:
\begin{center}
$f_{d-1}(x,y,z)=c\pm \frac{c_0 EI(\arcsin(\sqrt{\frac{z^2}{x^2+z^2}}),\frac{x^2+z^2}{y^2+z^2})G(x^2-y^2,x^2+z^2)}{\sqrt{y^2+z^2}}$,
\end{center}
where $EI$ stands for the elliptic integral of the first kind. Since $f_{d-1}$ should be a polynomial, we conclude that terms of order $d-1$ do not exist.
\end{proof}
We proceed to find the terms of order $d-2$, that is, we solve equation (\ref{eqn2solve2}) for $j=d-1$. Since $f_{d-1}=0$, it reads as:
\begin{center}
$yz\frac{\partial f_{d-2}}{\partial x}+xz\frac{\partial f_{d-2}}{\partial y}-xy\frac{\partial f_{d-2}}{\partial z} = 0$,
\end{center}
which is actually the same as equation (\ref{eqn2solve1}) we have already solved. Thus $f_{d-2}(x,y,z)=W(x^2-y^2,x^2+z^2)$, where $W:\mathbb{R}^2\rightarrow \mathbb{R}$ a homogeneous polynomial of degree $(d-2)/2$ (remember that, as we have shown, $d$ is an even number).

To find the terms of degree $d-3$ we solve equation (\ref{eqn2solve2}) for $j=d-2$. We conclude, once again, that there are no terms of degree $d-3$ and, in this way, we have proved the following:
\begin{lemma}
Darboux polynomials of vector field $X$ with constant cofactor are of the form:
$$f(x,y,z)=\sum _{j=0}^{d/2}G_j(x^2-y^2,x^2+z^2),$$
where $G_j:\mathbb{R}^2\rightarrow \mathbb{R}$ homogeneous polynomials of degree $2j$.
\end{lemma}
Actually, we can show something more:
\begin{lemma}
\label{eachterm}
Let $f(x,y,z)$ presented above be a Darboux polynomial of $X$, with cofactor $c_0$. Then, each term of $f(x,y,z)$ is also a Darboux polynomial for $X$, with the same cofactor.
\end{lemma}
\begin{proof}
We study, once again, equation $\mathcal{L}_Xf=c_0f$. We have:
\begin{center}
$\mathcal{L}_Xf=c_0f\Rightarrow \mathcal{L}_X(\sum G_j)=c_0\sum G_j\Rightarrow\sum (\mathcal{L}_XG_j)=c_0\sum G_j\Rightarrow$\\
$\Rightarrow \sum (\mathcal{L}_{X_1}G_j+\mathcal{L}_{X_2}G_j)=c_0\sum G_j$,
\end{center}
where $X_1,\ X_2$ are the linear and nonlinear parts of $X$, respectively. Since each $G_j$ is of degree $2j$, $\mathcal{L}_{x_1}G_j$ is of degree $2j$ and $\mathcal{L}_{X_2}G_j$ is of degree $2j+1$. We conclude that $\mathcal{L}_{X_2}G_j=0$ and thus $\mathcal{L}_{X}G_j=c_0G_j$, proving the statement.  
\end{proof}
We now recall, from \cite{Christopher}, the following:
\begin{lemma}
Let $f$ be a polynomial and $f=\prod _{i=1}^sf_i^{a_i}$ its decomposition into irreducible factors. The polynomial $f$ is a Darboux polynomial for a vector field if, and only if, $f_i$'s are also Darboux polynomials for the same field. Furthermore, $k=\sum _{i=1}^sa_ik_i$, where $k$ the cofactor of $f$ and $k_i$'s the cofactors of $f_i$'s.
\end{lemma}
Our goal is to find the basic Darboux polynomials for the vector field $X$. To make it more explicit, under the light of the previous lemma, we give the following definition.
\begin{definition}
Let $\textsl{Dar}$ be the set of all Darboux polynomials of a vector field. A finite collection $\mathcal{I}$ of such polynomials generate $\textsl{Dar}$ if every member of $\textsl{Dar}$ is the sum of members of $\mathcal{I}$ with the same cofactor, or a product of members of $\mathcal{I}$.  
\end{definition}
Using \ref{eachterm}, we see that every Darboux polynomial of $X$ can be constructed from the polynomials $x^2-y^2,\ x^2+z^2$, using the operations of addition and multiplication. Thus, we search for the generating Darboux polynomials among the linear combinations $m(x^2-y^2)+n(x^2+z^2),\ m,n\in \mathbb{R}$.

This combination should satisfy equation (\ref{eqn2solve2}), for
\begin{center}
$j=2, f_j(x,y,z)=m(x^2-y^2)+n(x^2+z^2),\ m,n\in \mathbb{R}$ and $f_{j-1}=0$.
\end{center}
We get:
\begin{center}
$(2am+2an)x^2-2bmy^2+2nz^2=c_0(mx^2-my^2+nx^2+nz^2)$,
\end{center}
which gives us the system:
\[
\begin{cases}
2am+2an=c_0m+c_0n\\
2bm=mc_0\\
2n=nc_0
\end{cases},
\] 
having the following solutions:
\begin{itemize}
\item{$n=0,\ a=b,\ c_0=2b$}
\item{$m=0,\ a=1,\ c_0=2$}
\item{$m=-n,\ b=1,\ c_0=2$}
\end{itemize}
Combining all the above, we have the main result of this section.
\begin{theorem}
Vector field $X=(ax+yz)\frac{\partial}{\partial x}+(by+xz)\frac{\partial}{\partial y}+(z-xy)\frac{\partial}{\partial z}$ possesses the following Darboux polynomials:
\begin{itemize}
\item {$H_b(x,y,z)=x^2+z^2$, for $a=1$ and $b\in \mathbb{R}$, having cofactor 2.}
\item {$H_a(x,y,z)=y^2+z^2$, for $b=1$ and $a\in \mathbb{R}$, having cofactor 2.}
\item {$H_{ab}(x,y,z)=x^2-y^2$, for $a=b\in \mathbb{R}$, having cofactor 2a.}
\end{itemize}
All other Darboux polynomials are generated from these three.
\end{theorem}
We have thus constructed the complete list of Darboux polynomials for the vector field of interest here.

Following the exact same steps, one may find the complete list of Darboux polynomials for vector field (\ref{d2}), in case the sign in the last equation in positive. We state this result here as well, omitting the proof.
\begin{theorem}
Consider vector field $(ax+yz)\frac{\partial}{\partial x}+(by+xz)\frac{\partial}{\partial y}+(z+xy)\frac{\partial}{\partial z}$. It possesses the following Darboux polynomials.
\begin{itemize}
\item {$K_b(x,y,z)=x^2-z^2$, for $a=1$ and $b\in \mathbb{R}$, having cofactor 2.}
\item {$K_a(x,y,z)=y^2-z^2$, for $b=1$ and $a\in \mathbb{R}$, having cofactor 2.}
\item {$K_{ab}(x,y,z)=x^2-y^2$, for $a=b\in \mathbb{R}$, having cofactor 2a.}
\end{itemize}
All other Darboux polynomials are generated from these three.
\end{theorem}
Having constructed the complete list of Darboux polynomials for the vector field of interest, we are now going to use these polynomials to comment on the global behaviour of this field.
\section{Global phase portraits}
\label{portraits}
In this section, we study the global behavior of the vector field $X$, using the Darboux polynomials constructed above.
\subsection{The $a=1,\ b\in \mathbb{R}$ ($b=1,\ a\in \mathbb{R}$) case}
We begin by considering the $a=1$ case. It turns out that the qualtative behaviour of the vector field is quite simple. We describe it in the next:
\begin{proposition}
Consider vector field $X$, with $a=1$. Then:
\begin{itemize}
\item {For $b>0$ there exists a unique fixed point, located at the origin. This fixed point is globally repelling}.
\item {For $b=0$ the whole $y$--axis consists of fixed points. All other orbits of the field tend to one of these fixed points, for $t\rightarrow-\infty$, and to infinity, for $t\rightarrow +\infty$.}
\item {For $b<0$ there exists a unique fixed point, located at the origin. This fixed point is hyperbolic, having a one--dimensional stable manifold, that is, the $y$--axis. All other orbits tend to infinity, for $t\rightarrow +\infty$.}
\end{itemize}
\end{proposition}
We omit the proof, since it is based on the fact that, for these parameter values, vector field $X$ possesses $H_b$ as a Darboux polynomial.
\begin{remark}
Note that in \cite{Anastassiou}, and specifically in the proof of Proposition ($1$), polynomial $H_b$ is used as a Darboux polynomial for $b=0$ and every positive value of $a$, although, as we have seen above, it is a Darboux polynomial only for $a=1$. However, the conclusions of that proposition still hold, since one may use the same function to prove that the distance of every regular orbit from the $y$--axis becomes infinite, as $t\rightarrow +\infty$.
\end{remark}
Using the existence of $H_a$, for $b=1$, one may prove the following: 
\begin{proposition}
Consider vector field $X$, with $b=1$. Then:
\begin{itemize}
\item {For $a>0$ there exists a unique fixed point, located at the origin. This fixed point is globally repelling}.
\item {For $a=0$ the whole $x$--axis consists of fixed points. All other orbits tend to one of these fixed points, for $t\rightarrow-\infty$, and to infinity, for $t\rightarrow +\infty$.}
\item {For $a<0$ there exists a unique fixed point, located at the origin. This fixed point is hyperbolic, having a one--dimensional stable manifold, that is, the $x$--axis. All other orbits tend to infinity, for $t\rightarrow +\infty$.}
\end{itemize}
\end{proposition}
\subsection{The $a=b$ case}
It turns out that the $a=b$ case is more complicated than the cases studied above. Once again, we have to distinguish between cases.
\begin{proposition}
Consider vector field $X$ for $a=b>0$. There exists a unique fixed point, located at the origin, which is globally repelling. 
\end{proposition}
\begin{proof}
Since $H_{a,b}$ is a Darboux polynomial having cofactor $2a>0$, surfaces $x^2-y^2=0\Rightarrow x=\pm y$ remain invariant under the flow of vector field $X$, while all other orbits escape to infinity.

The restriction of the vector field $X$ on the $x=y$ surface reads as $(ax+xz)\frac{\partial}{\partial x}+(z-x^2)\frac{\partial}{\partial z}$. It possesses a unique fixed point, located at the origin, which is globally repelling, since $V(x,z)=x^2+z^2$ serves as a Lyapunov function.

The restriction of the vector field $X$ on the $x=-y$ surface reads as $(ax-xz)\frac{\partial}{\partial x}+(z+x^2)\frac{\partial}{\partial z}$, which has the same behaviour as before. 
\end{proof}
\begin{proposition}
Consider vector field $X$ for $a=b=0$. Both the $x$ and the $y$ axes consist of fixed points. All other orbits tend to infinity, for $t\rightarrow +\infty$.  
\end{proposition}
\begin{proof}
The proof relies on the fact that, for $a=b=0$, $H_{ab}$ serves as first integral of the system, since it is a Darboux polynomial with zero cofactor.
\end{proof}
The most interesting case is the $a=b<0$, where we have the following:
\begin{proposition}
For $a=b<0$ vector field $X$ possesses five fixed points, namely the origin and: 
\begin{center}
$(\sqrt{-a},\sqrt{-a},a),\ (-\sqrt{-a},-\sqrt{-a},a),$\\ 
$(-\sqrt{-a},\sqrt{-a},a),\ (\sqrt{-a},-\sqrt{-a},a)$,
\end{center}
all of which are hyperbolic saddles. There exist four heteroclinic orbits, connecting each one of the non--trivial equilibria with the origin, contained in the invariant planes $x=\pm y$. All other orbits tend to these two planes for $t\rightarrow +\infty$.
\end{proposition}
The $a=b<0$ is thus the most interesting one. In \cite{Anastassiou} it is shown that this case serves as the starting point, in the parameter space of the system, of a path leading to chaotic behaviour, which deserves a more systematic study.  
\section{Conclusions}
The theory of Darboux integrability is a well established one and of particular importance since it provides us with results which enlighten the 
global behaviour of vector fields (see \cite{Dumortier2006} for details). Therefore, a number of systems have been studied from this perspective.

Let us focus here on articles \cite{Swinnerton,Llibre-Zhang}. In these articles the Darboux integrability of two famous systems was studied, namely those due to Lorenz and R\"{o}ssler. These systems present different symmetric properties, since the Lorenz system possesses a two--fold symmetry while the R\"{o}ssler system possesses no symmetry at all. We have here studied the Darboux polynomials of a system with richer symmetries than those of Lorenz.

A natural question arises: can one determine the symmetric properties of a vector field from the algebraic properties of the module generated from the Darboux polynomials of the system? What other global properties can be deduced from the structure of this module? Let us note here that the structure of vector fields having a given set of polynomials as Darboux polynomials has already attracted the attention of researchers (see \cite{Pantazi1,Pantazi2}) and deserves much more attention.

We hope to further comment on these issues in a future publication.
\section*{Acknowledgements}
The authors would like to express their gratitude to the anonymous referee, who spotted a critical gap in an earlier version of our work. His comments led us to correct and significantly simplify our calculations. We would also like to thank the editors for effectively helping us during the preparation of the manuscript.    

\end{document}